\documentclass[11pt]{amsart}
\usepackage{amsmath}
\usepackage{amsthm}
\usepackage{graphicx}

\newtheorem{theorem}{Theorem}
\newtheorem{proposition}[theorem]{Proposition}
\newtheorem{lemma}[theorem]{Lemma}
\newtheorem{corollary}[theorem]{Corollary}

\theoremstyle{remark}
\newtheorem*{remark}{Remark}

\theoremstyle{definition}

 \def\q{\mathbb{Q}}
 \def\v{{\rm v}}
 
 \def\b{\mathcal{B}(A_n)}

\def\t{\mathcal{S}(A_n)}
\def\gr{\mathcal{G}}
  \def\base{\mathcal{B}}

\begin{document}
\title[Root polytopes, triangulations, and the subdivision algebra, I]{Root polytopes, triangulations, and the subdivision algebra, I}
\author{Karola M\'esz\'aros }
\address{
Department of Mathematics, Massachusetts Institute of Technology, Cambridge, MA 02139
}
\date{June 28, 2009}

\begin{abstract}
 
  The  type $A_{n}$ root polytope $\mathcal{P}(A_{n}^+)$ is the convex hull in $\mathbb{R}^{n+1}$ of the origin and the points $e_i-e_j$ for $1\leq i<j \leq n+1$.  Given a tree $T$ on the vertex set $[n+1]$, the associated   root polytope $\mathcal{P}(T)$    is the intersection of  $\mathcal{P}(A_{n}^+)$ with the cone generated by the  vectors  $e_i-e_j$, where $(i, j) \in E(T)$, $i<j$.   The reduced forms of a certain monomial $m[T]$ in commuting variables $x_{ij}$   under the reduction $x_{ij}x_{jk} \rightarrow x_{ik}x_{ij}+x_{jk}x_{ik}+\beta x_{ik}$,  can be interpreted as triangulations of $\mathcal{P}(T)$.  Using these triangulations, the volume and Ehrhart polynomial  of $\mathcal{P}(T)$ are obtained.  If we allow variables $x_{ij}$ and $x_{kl}$ to commute  only  when   $i, j, k, l$  are distinct, then the reduced form of $m[T]$ is unique and yields  a canonical triangulation of  $\mathcal{P}(T)$  in which each simplex corresponds to a noncrossing alternating forest.  Most generally, the reduced forms of  all monomials in the noncommutative case are unique. 
  
  \end{abstract}

\maketitle
  
  \section{Introduction}
\label{sec:in}
 
  In this paper we develop the connection between triangulations of type $A_n$ root polytopes and two closely related algebras: the subdivision algebra $\t$ and the algebra $\b$, which we call the quasi-classical Yang-Baxter algebra following A. N. Kirillov. 
 The close connection of the root polytopes and the algebras $\t$ and $\b$ is displayed by the variety of results this connection yields: both in the realm of polytopes and in the realm of the algebras. Two closely related algebras with tight connections to Schubert calculus have been studied by Fomin and Kirillov in \cite{fk} and by Kirillov in \cite{k1}. Before stating definitions and reasons, we pause at   Exercise 6.C6 of Stanley's Catalan Addendum \cite{cat} to learn the following. 
   
 Consider the monomial $w=x_{12}x_{23}\cdots x_{n,n+1}$ in commuting   variables $x_{ij}$.  Starting with $p_0=w$, produce a sequence of polynomials $p_0, p_1, \ldots, p_m$ as follows.  To obtain $p_{r+1}$ from $p_r$,  choose a term of   $p_r$ which  is divisible by $x_{ij}x_{jk}$, for some $i,j,k$, and replace the factor  $x_{ij}x_{jk}$ in this term   with   $x_{ik}(x_{ij}+x_{jk})$. Note that   $p_{r+1}$   has one more term than $p_r$. Continue  this process until   a  polynomial  $p_m$  is obtained, in which  no term is divisible by  $x_{ij}x_{jk}$, for any $i,j,k$.  Such a polynomial $p_m$  is a {\bf reduced form} of $w$. Exercise 6.C6 in \cite{cat} states that, remarkably, while the reduced form is not unique, it turns out  that the number of terms in a reduced form is  always the {\bf Catalan number}  $C_n=\frac{1}{n+1} {2n \choose n}$.

The angle from which we look at this problem gives a perspective reaching far beyond its  setting in the world of polynomials. On one hand, the reductions can be interpreted in terms of root polytopes and their subdivisions,  yielding a geometric, and subsequently a combinatorial,  interpretation of reduced forms. On the other hand, using the combinatorial results obtained about the reduced forms, we obtain a method for  calculating the volumes and Ehrhart polynomials of a family of  root polytopes. 
 
Root polytopes were defined  by  Postnikov in \cite{p1}. The  full root polytope  $\mathcal{P}(A_{n}^+)$, which is   the convex hull in $\mathbb{R}^{n+1}$ of the origin and points $e_i-e_j$ for $1\leq i<j \leq n+1$, already made an appearance in the work of Gelfand, Graev and Postnikov \cite{GGP}, who gave a canonical triangulation of it  in terms of noncrossing alternating trees on $[n+1]$.  We obtain canonical triangulations for all acyclic root polytopes, of which $\mathcal{P}(A_{n}^+)$ is a special case. 

We define  \textbf{acyclic root polytopes}  $\mathcal{P}(T)$   for a tree $T$ on the vertex set $[n+1]$  as the intersection of  $\mathcal{P}(A_{n}^+)$ with a cone generated by the  vectors  $e_i-e_j$, where  $(i, j) \in E(T)$, $i<j$.      Let  
$$\overline G=([n+1], \{(i, j) \mid  \mbox{there exist edges $(i, i_1) \ldots, (i_k, j)$ in $G$ such that}$$ 

$i<i_1<\ldots<i_k<j\}),$ 

\noindent denote the \textbf{transitive closure} of the graph $G$. Recall that a graph $G$ on the vertex set $[n+1]$ is said to be {\bf noncrossing} if there are no vertices $i <j<k<l$ such that $(i, k)$ and $(j, l)$ are edges in $G$. 
A graph $G$ on the vertex set $[n+1]$ is said to be {\bf alternating} if there are no vertices $ i <j<k $ such that $(i, j)$ and $(j, k)$ are edges in $G$.  Alternating trees were introduced in \cite{GGP}. Gelfand, Graev and Postnikov \cite{GGP} showed that the number of noncrossing alternating trees on $[n+1]$ is counted by the Catalan number $C_n$.

\begin{theorem} \label{vol}
If $T$ is a noncrossing tree on the vertex set $[n+1]$ and   $T_1,\ldots, T_k$ are the noncrossing alternating spanning trees of  $\overline T$, then the root polytopes $\mathcal{P}(T_1), \ldots, \mathcal{P}(T_k)$ are   $n$-dimensional simplices with disjoint interiors whose union is  $\mathcal{P}(T)$. Furthermore, 

 $$\mbox{\em  vol}\, \mathcal{P}(T)=f_{T} \frac{1}{n!},$$ where  $f_{T} $ denotes the number of noncrossing alternating spanning trees of  $\overline T$.
\end{theorem}

 Theorem \ref{vol}  can be generalized in a few directions. We calculate the Ehrhart polynomial of $\mathcal{P}(T)$; see Sections \ref{sec:forest} and \ref{sec:gen}.  We   describe the intersections of the top dimensional simplices   $\mathcal{P}(T_1), \ldots, \mathcal{P}(T_k)$  in Theorem \ref{vol} in terms of   noncrossing alternating spanning forests of $\overline T$  in Section \ref{sec:gen}. Theorem \ref{vol} and its generalizations can also be proved for any forest $F$, not necessarily noncrossing, as explained in Section \ref{sec:shelling}.  In Section \ref{sec:shelling} we also prove that the triangulation in Theorem \ref{vol} is shellable, and provide a second method for calculating the Ehrhart polynomial of  $\mathcal{P}(T)$.
 
 The proof of Theorem \ref{vol} relies on   relating the triangulations of a root polytope $\mathcal{P}(T)$ to reduced forms of a monomial $m[T]$ in variables $x_{ij}$, which we now define.   Let $\t$ and $\b$ be two  associative algebras   over the polynomial ring  $\mathbb{Q}[\beta]$, where $\beta$ is a variable (and a central element), generated by the set of elements $\{x_{ij} \mid 1 \leq i <j\leq n+1\}$ modulo the relation   $x_{ij}x_{jk}=x_{ik}x_{ij}+x_{jk}x_{ik}+\beta x_{ik}$. The {\bf subdivision algebra} $\t$ is commutative, i.e., it has additional relations $x_{ij}x_{kl}= x_{kl}x_{ij}$ for all $i ,j, k, l$, while $\b$, which we call  the {\bf quasi-classical Yang-Baxter algebra} following Kirillov \cite{kir},  is noncommutative and has additional relations $x_{ij}x_{kl}= x_{kl}x_{ij}$ for $i, j, k, l$ distinct only. The motivation for calling $\t$ the subdivision algebra is simple; the relations of $\t$ yield certain subdivisions of root polytopes, which we explicitly demonstrate by the Reduction Lemma (Lemma \ref{reduction_lemma}).

We treat the first relation as a  \textbf{reduction rule}:

\begin{equation} \label{red}
x_{ij}x_{jk}\rightarrow x_{ik}x_{ij}+x_{jk}x_{ik}+\beta x_{ik}.
 \end{equation}
 
 A  \textbf{reduced form} of the monomial $m$ in the algebra $\t$ (algebra $\b$) is a polynomial $P_n^\mathcal{S}$ (polynomial $P_n^\mathcal{B}$) obtained by successive applications of reduction (\ref{red}) until no further reduction is possible, where we allow commuting any two variables (commuting any two variables $x_{ij}$ and $x_{kl}$ where $i, j, k, l$ are distinct)   between reductions. Note that the reduced forms are not necessarily unique.
  
A possible sequence of reductions in algebra $\t$ yielding a reduced form of $x_{12}x_{23}x_{34}$ is given by

\begin{eqnarray} \label{ex1}
 x_{12} \mbox {\boldmath$  x_{23}x_{34}$} & \rightarrow & \mbox{\boldmath$x_{12}$}x_{24}\mbox{\boldmath$x_{23}$}+\mbox{\boldmath$x_{12}$}x_{34}\mbox {\boldmath$x_{24}$}+\beta \mbox {\boldmath$x_{12}x_{24}$} \nonumber \\
& \rightarrow& \mbox {\boldmath$ x_{24}$} x_{13}\mbox {\boldmath$x_{12}$}+x_{24}x_{23}x_{13}+  \beta x_{24}x_{13}+x_{34}x_{14}x_{12}+x_{34}x_{24}x_{14} \nonumber \\
& &+\beta x_{34}x_{14}+\beta x_{14}x_{12}+\beta x_{24}x_{14}+\beta^2 x_{14} \nonumber \\
& \rightarrow  &x_{13}x_{14}x_{12}+x_{13}x_{24}x_{14}+\beta x_{13}x_{14}+x_{24}x_{23}x_{13}+\beta x_{24}x_{13}\nonumber \\
& & +x_{34}x_{14}x_{12}+x_{34}x_{24}x_{14}+\beta x_{34}x_{14}+\beta x_{14}x_{12}+\beta x_{24}x_{14}\nonumber \\
& &+\beta^2 x_{14}
\end{eqnarray}

\noindent where the pair of variables on which the reductions are performed is in boldface. The reductions are performed on each monomial separately.

Some of the reductions performed above are not allowed in the noncommutative algebra $\b$. The following is an  example of how to reduce $x_{12}x_{23}x_{34}$ in the noncommutative case.

  \begin{eqnarray}  \label{ex2}
 x_{12} \mbox {\boldmath$  x_{23}x_{34}$} & \rightarrow&  \mbox{\boldmath$x_{12}x_{24}$} x_{23}  +\underline{x_{12}x_{34}} x_{24} +\beta \mbox {\boldmath$x_{12}x_{24}$} \nonumber \\
&\rightarrow& x_{14}\mbox{\boldmath $x_{12}x_{23}$}+x_{24}x_{14}x_{23}+\beta x_{14}x_{23}+ x_{34}\mbox{\boldmath $x_{12}x_{24}$}+\beta x_{14}x_{12}\nonumber \\ &&+\beta x_{24}x_{14}+\beta^2 x_{14}
\nonumber 
\\
&\rightarrow & x_{14}x_{13}x_{12}+x_{14}x_{23}x_{13}+\beta x_{14}x_{13}+x_{24}x_{14}x_{23}+\beta x_{14}x_{23}\nonumber \\ &&+x_{34}x_{14}x_{12}+x_{34}x_{24}x_{14}+\beta x_{34}x_{14}+\beta x_{14}x_{12}+\beta x_{24}x_{14}\nonumber \\ &&+\beta^2 x_{14}
\end{eqnarray}

In the example above the pair of variables on which the reductions are performed is in boldface, and the variables which we commute are underlined.

The ``reason" for allowing $x_{ij}$ and $x_{kl}$ to commute only  when $i, j, k, l$ are distinct   might not be apparent at first, but as we prove in Section \ref{sec:gen},  it insures that, unlike in the  commutative case, there are unique reduced forms    for a natural set of monomials. Kirillov \cite{kir} observed  that the monomial $w=x_{12}x_{23}\cdots x_{n,n+1}$ has a unique reduced form in the  quasi-classical Yang-Baxter algebra  $\b$, and asked for a bijective proof. The uniqueness of the reduced form of $w$ is a special case of our   results, and the desired bijection follows  from our proof methods. 

Before we can state a simplified version of our main result on reduced forms, we need one more piece of notation.    Given a graph $G$ on the vertex set $[n+1]$ we associate to it the monomial $m^\mathcal{S}[G]=\prod_{(i, j) \in E(G)}x_{ij}$; if $G$ is edge-labeled with labels $1, \ldots, k$, we can also associate to it the noncommutative monomial $m^\mathcal{B}[G]=\prod_{a=1}^k x_{i_a, j_a}$, where $E(G)=\{ (i_a, j_a)_a \mid a \in [k]\}$ and $(i, j)_a$  denotes an   edge  $(i , j )$ labeled $a$. 
 
\begin{theorem} \label{thm1}
Let $T$ be a noncrossing tree on the vertex set $[n+1]$,  and  $P^\mathcal{S}_n$ a reduced form of $m^\mathcal{S}[T]$. Then,  $$P_n^\mathcal{S}(x_{ij}=1,  \beta=0)=  f_{T},$$ where $f_{T} $ denotes the number of noncrossing alternating spanning trees of $\overline T$.

If we label the edges of  $T$ so that it becomes a good tree   (to be defined in Section \ref{reductionsB}), then  the reduced form   $P^\mathcal{B}_n$  of the monomial  $m^\mathcal{B}[T]$ is

$$P^\mathcal{B}_n(x_{ij},  \beta=0)=\sum_{T_0}  x^{T_0},$$

\noindent where the sum runs over all   noncrossing alternating spanning trees $T_0$ of $\overline T$ with lexicographic  edge-labels (to be defined in Section \ref{reductionsB1}) and $x^{T_0}$ is defined to be the noncommutative monomial $\prod_{l=1}^n x_{i_l,j_l}$ if $T_0$ contains the edges $(i_1, j_1)_1,$ $ \ldots, (i_n, j_n)_n$.  

\end{theorem}

We generalize Theorem \ref{thm1} for any  $\beta$; see Sections \ref{sec:red} and \ref{sec:gen}.  Theorem \ref{thm1} can also be generalized for any forest $F$; see Sections \ref{sec:forest} and \ref{sec:shelling}.  Finally, we prove using noncommutative Gr\"obner bases techniques that:

\begin{theorem}
The reduced form $P^\mathcal{B}_n$ of any monomial $m$ is unique, up to commutations.
\end{theorem}

 This paper  is organized as follows.   In Section \ref{sec:red}  we reformulate the reduction process in terms of graphs  and elaborate further on Theorem \ref{thm1} and its generalizations. In Section \ref{sec:root} we discuss acyclic root polytopes and relate them to  reductions via the Reduction Lemma. We prove the Reduction Lemma,  which translates reductions into polytope-language,  in Section \ref{sec:reduction_lemma}. In Section \ref{sec:forest} we use the Reduction Lemma to prove general theorems about reduced forms of $m^\mathcal{S}[F]$, the volume and Ehrhart polynomial of $\mathcal{P}(F)$, for any forest $F$.  The lemmas of  Section \ref{reductionsB}    indicate the significance of considering reduced forms in the noncommutative algebra $\b$.  In Section \ref{reductionsB1} we  prove Theorems \ref{vol} and \ref{thm1} for a special tree $T$. Theorems \ref{vol} and \ref{thm1} as well as their generalizations are proved in Section \ref{sec:gen}. In Section \ref{sec:shelling} we shell the canonical triangulation described in Theorem \ref{vol}, and provide an alternative way to obtain the Ehrhart polynomial of $\mathcal{P}(T)$ for a tree $T$. We conclude in Section \ref{sec:grobi} by proving that the reduced form $P^\mathcal{B}_n$ of any monomial $m$ is unique using noncommutative Gr\"obner bases techniques.
    
   \section{Reductions in terms of graphs}
\label{sec:red}


 
We can  phrase the reduction process described in Section \ref{sec:in}  in terms of graphs. This view will be useful throughout the paper.  Think of a monomial $m \in A$  as a directed graph $G$ on the vertex set $[n+1 ]$ with an edge directed from $i$ to $j$ for each appearance of $x_{ij}$ in $m$. Let $G^\mathcal{S}[m]$ denote this graph. If, however, we are in the noncommutative version of the problem, and   $m=\prod_{l=1}^p x_{i_l, j_l}$, then  we can think of $m$ as a directed graph $G$ on the vertex set $[n+1 ]$ with $p$ edges labeled $1, \ldots, p$, such that the edge labeled $l$ is directed from vertex $i_l$ to $j_l$.  Let $G^\mathcal{B}[m]$ denote the  edge-labeled graph just described.
  Let $(i, j)_a$  denote an   edge  $(i , j )$ labeled $a$. 
   It is straighforward to reformulate the  reduction rule (\ref{red}) in terms of reductions on graphs. If $m \in A$, then it reads as follows.

    The {\bf reduction rule for graphs:} Given   a graph $G_0$ on the vertex set $[n+1]$ and   $(i, j), (j, k) \in E(G_0)$ for some $i<j<k$, let   $G_1, G_2, G_3$ be graphs on the vertex set $[n+1]$ with edge sets
  \begin{eqnarray} \label{graphs}
E(G_1)&=&E(G_0)\backslash \{(j, k)\} \cup \{(i, k)\}, \nonumber \\
E(G_2)&=&E(G_0)\backslash \{(i, j)\} \cup \{(i, k)\},\nonumber \\ 
E(G_3)&=&E(G_0)\backslash \{(i, j)\} \backslash \{(j, k)\} \cup \{(i, k)\}. 
\end{eqnarray}

    We say that $G_0$ \textbf{reduces} to $G_1, G_2, G_3$ under the reduction rules defined by equations (\ref{graphs}).
 
 The reduction rule for graphs $G^\mathcal{B}[m]$ with $m \in B$ is  explained in Section \ref{reductionsB}.

An \textbf{$\mathcal{S}$-reduction tree} $\mathcal{T}^\mathcal{S}$  for a monomial $m_0$, or equivalently, the graph $G^\mathcal{S}[m_0]$,  is constructed as follows. 
  The root of $\mathcal{T}^\mathcal{S}$ is labeled by $G^\mathcal{S}[ m_0]$. Each node $G^\mathcal{S}[m]$ in $\mathcal{T}^\mathcal{S}$  has three children, which depend on the choice of the edges  of  $G^\mathcal{S}[m]$ on which we perform the reduction. Namely, if the reduction  is performed on edges $(i, j), (j, k) \in E(G^\mathcal{S}[m])$,  $i<j<k$, then the three children of  the node $G_0=G^\mathcal{S}[m]$   are labeled by the graphs   $G_1, G_2, G_3$   as described by equation (\ref{graphs}). For an example of an $\mathcal{S}$-reduction tree; see Figure \ref{fig} (disregard the edge-labels).

Summing the monomials to which the graphs labeling the  leaves  of the reduction tree $\mathcal{T}^\mathcal{S}$  correspond multiplied by suitable powers of $\beta$, we obtain a reduced form of  $m_0$.

\begin{figure}[htbp] 
\begin{center} 
\includegraphics{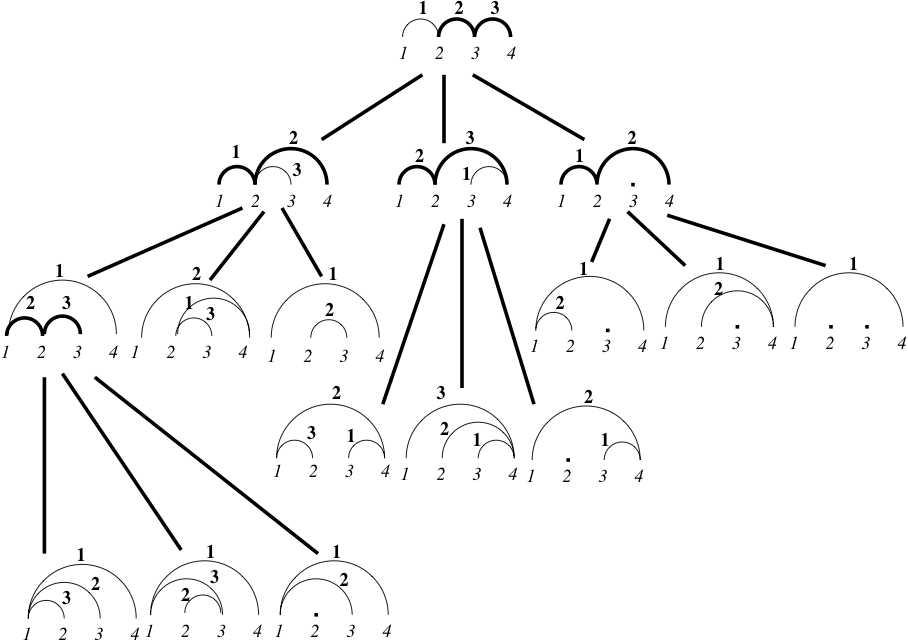} 
\caption{This is an $\mathcal{S}$-reduction tree with root labeled by $G^\mathcal{S}[x_{12}x_{23}x_{34}]$, when the edge-labels are disregarded. The boldface edges indicate where the reduction is performed. We can read off the following  reduced form of $x_{12}x_{23}x_{34}$ from the set of leaves: $x_{14}x_{13}x_{12}+x_{14}x_{23}x_{13}+\beta x_{14}x_{13}+x_{24}x_{14}x_{23}+\beta x_{14}x_{23} +x_{34}x_{14}x_{12}+x_{34}x_{24}x_{14}+\beta x_{34}x_{14}+\beta x_{14}x_{12}+\beta x_{24}x_{14} +\beta^2 x_{14}$. When the edge-labels are taken into account, this is the $\mathcal{B}$-reduction tree  corresponding to equation (\ref{ex2}).   Note that in the second child of the root we commuted edge-labels 1 and 2.} 
\label{fig} 
\end{center} 
\end{figure} 

Let $T$ be  a noncrossing tree on the vertex set $[n+1]$. In terms of reduction trees, Theorem \ref{thm1} states that the number of  leaves labeled by graphs with exactly $n$ edges of an $\mathcal{S}$-reduction tree  with root  labeled $T$     is independent of the particular $\mathcal{S}$-reduction tree. The generalization of Theorem \ref{thm1} for any $\beta$ states that  the number leaves labeled by graphs with exactly $k$ edges of an $\mathcal{S}$-reduction tree with root  labeled $T$,  is  independent of the particular $\mathcal{S}$-reduction tree for any $k$.  In terms of reduced forms we can write this as follows. 
  If $P_n^\mathcal{S}$ is the reduced form of a monomial $m^\mathcal{S}[T]$ for a noncrossing tree $T$, then  

 $$P_n^\mathcal{S}(x_{ij}=1)= \sum_{m=0}^{n-1} f_{T, n-m} \beta^m,$$ where 
 $f_{T, k} $ denotes the number of noncrossing alternating spanning forests of $\overline T$ with $k$ edges and   additional technical  requirements detailed in Section \ref{sec:gen}. Also,  if $P_n^\mathcal{B}$ is the reduced form of a monomial $m^\mathcal{B}[T]$ for a noncrossing good tree $T$   (defined in Section \ref{reductionsB}), then   $$P^\mathcal{B}_n(x_{ij})=\sum_F  x^F,$$  where the sum runs over all  noncrossing alternating spanning forests $F$  of $\overline T$ with lexicographic  edge-labels  (defined in Section \ref{reductionsB1}) and additional  technical  requirements  detailed in Section \ref{sec:gen}. 
 
If we   consider the reduced forms of the path monomial $w=\prod_{i=1}^n x_{i, i+1}$, then  $T=P=([n+1], \{(i, i+1) \mid i \in [n] \})$, and $f_{P, k}$  is simply  the number of noncrossing alternating spanning forests on $[n+1]$ with $k$ edges containing edge $(1, n+1)$. Furthermore,
 $P^\mathcal{B}_n(x_{ij})=\sum_F  x^F,$   where the sum runs over all   noncrossing alternating spanning forests $F$ on $[n+1]$ with lexicographic  edge-labels and  containing edge $(1, n+1)$. See Section \ref{reductionsB1} for the treatment of this special case.

\section{Acyclic root polytopes}
 \label{sec:root}
 
In the terminology of \cite{p1}, a root polytope of   type $A_{n}$ is  the convex hull of the origin and some of the points $e_i-e_j$ for $1\leq i<j \leq n+1$, where $e_i$ denotes the $i^{th}$ coordinate vector in $\mathbb{R}^{n+1}$. A very special root polytope is the full root polytope  $$\mathcal{P}(A_{n}^+)=\textrm{ConvHull}(0,  e_{ij}^- \mid  1\leq i<j \leq n+1),$$  where  $e_{ij}^-=e_i-e_j$. We study a  class of root polytopes including $\mathcal{P}(A_{n}^+)$, which we now discuss. 

Let $G$ be a   graph on the vertex set $[n+1]$.  Define $$\mathcal{V}_G=\{e_{ij}^- \mid  (i, j) \in E(G), i<j\}, \mbox{ a set of vectors associated to $G$;}$$

 $$\mathcal{C}(G)=\langle \mathcal{V}_G \rangle :=\{\sum_{ e_{ij}^- \in \mathcal{V}_G}c_{ij} e_{ij}^- \mid  c_{ij}\geq 0\}, \mbox{ the cone associated to $G$; and } $$  
  $$\overline{\mathcal{V}}_G=\Phi^+ \cap \mathcal{C}(G), \mbox{ all the positive roots of type $A_n$ contained in $\mathcal{C}(G)$}, $$
   where $\Phi^+=\{e_{ij}^-  \mid1\leq i<j \leq n+1\}$ is the set of         positive roots of type $A_n$. The idea to consider the positive roots of a root system inside a cone appeared earlier in Reiner's work \cite{R1}, \cite{R2} on signed posets. 
 
The root polytope $\mathcal{P}(G)$ associated to graph $G$ is  
 
 \begin{equation} \label{eq1} \mathcal{P}(G)=\textrm{ConvHull}(0, e_{ij}^- \mid e_{ij}^- \in \overline{\mathcal{V}}_G)\end{equation} The root polytope $\mathcal{P}(G)$ associated to graph $G$ can also be defined as \begin{equation} \label{eq2} \mathcal{P}(G)=\mathcal{P}(A_n^+) \cap \mathcal{C}(G).\end{equation} The equivalence of these two definition is proved in Lemma \ref{equivalent} in Section \ref{sec:reduction_lemma}. 
  
 Note that $\mathcal{P}(A_{n}^+)=\mathcal{P}(P)$ for the path graph  $P=([n+1], \{(i, i+1) \mid i \in [n]\}).$  While the choice of $G$ such that  $\mathcal{P}(A_{n}^+)=\mathcal{P}(G)$ is not unique, it becomes unique if we require that $G$ is {\bf minimal}, that is for no edge $(i, j) \in E(G)$ can the corresponding vector $e_{ij}^-$ be written as a nonnegative linear combination of the vectors corresponding to the   edges $E(G) \backslash \{e\}$. Graph $P$ is minimal. 
 
 We can describe the vertices in $ \overline{\mathcal{V}}_G$ in terms of paths in $G$. 
     A \textbf{playable route} of a graph $G$ is an ordered sequence of edges $(i_1, j_1), (i_2, j_2), \ldots, (i_l, j_l) \in E(G)$ such that   $i_1<j_1=i_2<j_2 \ldots j_{l-1}=i_l<j_l$.
 
 \begin{lemma} \label{playable_route}
 Let $G$ be a graph on the vertex set $[n+1]$. Any $v \in  \overline{\mathcal{V}}_G$ is $v=e_{i_1}-e_{j_l}$ for some playable route  $(i_1, j_1), (i_2, j_2), \ldots, (i_l, j_l)$ of $G$. If in addition $G$ is   acyclic, then the correspondence between playable routes of $G$ and vertices in $\overline{\mathcal{V}}_G$ is a bijection.
 \end{lemma}
 
 The proof of Lemma  \ref{playable_route} is straightforward, and  is left   to the reader.

 Define 
    $$\mathcal{L}_n=\{ G=([n+1], E(G))    \mid  \mbox{ $G$ is an acyclic graph}\},$$ and 
    $$\mathcal{L}(A_{n}^+)=\{ \mathcal{P}(G) \mid  G \in \mathcal{L}_n  \}, \mbox{the set of \textbf{acyclic root polytopes}.}$$        
    
    Note that the condition that $G$ is an acyclic graph is equivalent to  $\mathcal{V}_G$ being  a set of linearly independent vectors. 
    
    The full root polytope  $ \mathcal{P}(A_{n}^+) \in \mathcal{L}(A_{n}^+)$, since   the path graph $ P$ is acyclic. We show below  how to obtain central triangulations for all polytopes $\mathcal{P}   \in  \mathcal{L}(A_{n}^+)$. A \textbf{central triangulation} of a $d$-dimensional  root polytope $\mathcal{P}$ is a collection  of $d$-dimensional simplices with disjoint interiors whose union is  $\mathcal{P}$, the vertices of which are vertices of $\mathcal{P}$ and the origin is a vertex of all of them. Depending on the context we at times   take the intersections of these maximal simplices to be part of the triangulation.  


    We now  state the crucial lemma which relates root polytopes and algebras $\t$ and $\b$ defined in Section \ref{sec:in}.

    \begin{lemma} \label{reduction_lemma} \textbf{(Reduction Lemma)} 
Given   a graph $G_0 \in \mathcal{L}_n$ with $d$ edges let  $(i, j), (j, k) \in E(G_0)$ for some $i<j<k$ and $G_1, G_2, G_3$ as described by equations (\ref{graphs}).   Then  $G_1, G_2, G_3 \in \mathcal{L}_n$,
$$\mathcal{P}(G_0)=\mathcal{P}(G_1) \cup \mathcal{P}(G_2)$$   where all polytopes  $\mathcal{P}(G_0), \mathcal{P}(G_1), \mathcal{P}(G_2)$ are   $d$-dimensional and    
$$\mathcal{P}(G_3)=\mathcal{P}(G_1) \cap \mathcal{P}(G_2)  \mbox{   is $(d-1)$-dimensional. } $$
\end{lemma}
      \medskip
      
       What the Reduction Lemma really says is that performing a reduction on graph $G_0 \in \mathcal{L}_n$  is the same as ``cutting" the $d$-dimensional  polytope  $\mathcal{P}(G_0)$ into two $d$-dimensional polytopes $\mathcal{P}(G_1)$ and $ \mathcal{P}(G_2)$, whose vertex  sets are  subsets of the vertex set of   $\mathcal{P}(G_0)$, whose interiors are disjoint, whose union is $\mathcal{P}(G_0)$, and whose intersection is a facet of both. We prove the Reduction Lemma in Section \ref{sec:reduction_lemma}.

\section{ The proof of the Reduction Lemma}
\label{sec:reduction_lemma}

This section is devoted to proving the Reduction Lemma (Lemma \ref{reduction_lemma}).  As we shall see in Section \ref{sec:forest}, the Reduction  Lemma is the ``secret force"  that makes   everything fall into its place for acyclic root polytopes. We start by providing a simple lemma which characterizes the root polytopes which are simplices, then in  Lemma \ref{equivalent} we prove that equations (\ref{eq1}) and (\ref{eq2}) are equivalent definitions for the root polytope $\mathcal{P}(G)$,  and finally we prove the Cone Reduction Lemma (Lemma \ref{cone_reduction_lemma}), which, together with Lemma \ref{equivalent} implies the Reduction Lemma.

Lemma \ref{simplex} is implied by the results in \cite[Lemma 13.2]{p1}, but for the sake of completeness   we provide a proof of it. Note that the exact definitions and notations in \cite{p1} are different from ours. The idea for part of the proof of Lemma \ref{equivalent} appears in \cite{p1, fong} with different purposes. 
  
  \begin{lemma} \label{simplex}  (Cf. \cite[Lemma 13.2]{p1})
 For a  graph $G$ on $[n+1]$ vertices and $d$ edges, the polytope $\mathcal{P}(G)$ is a simplex if and only if $G$ is alternating and acyclic. If  $\mathcal{P}(G)$ is a simplex, then its $d$-dimensional normalized volume  $\mbox{\em vol}_d \,\mathcal{P}(G)=\frac{1}{d!}.$
  \end{lemma}
                   
   \proof
   It follows from equation (\ref{eq1}) that for a minimal graph $G$ the polytope  $\mathcal{P}(G)$ is a simplex if and only if  the vectors corresponding to the edges of $G$ are linearly independent and  $\mathcal{C}(G) \cap \Phi^+=\mathcal{V}_G$.  
   
  The vectors corresponding to the edges of $G$ are linearly independent if and only if  $G$ is acyclic. 
  By Lemma \ref{playable_route},   $\mathcal{C}(G) \cap \Phi^+=\mathcal{V}_G$  if and only if  $G$ contains no edges $(i, j), (j, k)$ with $i<j<k$, i.e. $G$ is alternating. 

That $\mbox{\rm vol}_d\, \mathcal{P}(G)=\frac{1}{d!}$ follows from the unimodality of $\Phi^+$.

   \qed                

\begin{lemma} \label{equivalent}
For any  graph $G$ on the vertex set $[n+1]$,  
 
 $$\emph{ConvHull}(0, e_{ij}^- \mid e_{ij}^- \in \overline{\mathcal{V}}_G)=\mathcal{P}(A_n^+) \cap \mathcal{C}(G).$$ 
\end{lemma}

\proof For a graph $H$ on the vertex set $[n+1]$, let $\sigma(H)=\textrm{ConvHull}(0, e_{ij}^- \mid (i, j) \in H, i<j)$. Then, by Lemma \ref{playable_route}, $\sigma(\overline{G} )=\textrm{ConvHull}(0, e_{ij}^- \mid e_{ij}^- \in \overline{\mathcal{V}}_G)$. Let $\sigma(\overline{G} )$ be a $d$-dimensional polytope for some $d \leq n$ and consider any central triangulation of it:  $\sigma(\overline{G} )=\cup_{F \in \mathcal{F}}\sigma(F)$, where $\{\sigma(F)\}_{F \in \mathcal{F}}$ is a set of $d$-dimensional simplices with disjoint interiors, $E(F) \subset E(\overline{G})$, $F \in \mathcal{F}$.  Since $\sigma(\overline{G} )=\cup_{F \in \mathcal{F}}\sigma(F)$ is a central triangulation, it follows that $\sigma(F)=\sigma(\overline{G} ) \cap \mathcal{C}(F)$, for $F \in \mathcal{F}$, and $\mathcal{C}(G)=\cup_{F \in \mathcal{F}}  \mathcal{C}(F)$.

Since $\sigma(F)$, $F \in \mathcal{F}$, is a  $d$-dimensional simplex, it follows that $F$ is a forest with $d$ edges. Furthermore, $F \in \mathcal{F}$ is an alternating forest, as otherwise $(i, j), (j, k) \in E(F) \subset E(\overline{G})$, for some $i<j<k$ and while $e_{ik}^-=e_{ij}^-+e_{jk}^- \in \sigma(\overline{G} ) \cap \mathcal{C}(F)$, $e_{ik}^- \not \in \sigma(F)$, contradicting that  $\cup_{F \in \mathcal{F}}\sigma(F)$ is a central triangulation of $\sigma(\overline{G} )$. Thus,  $\overline{F}=F$, and  $\sigma(F)=\sigma(\overline{F})$. It is clear that  
 $\sigma(\overline{F})=\textrm{ConvHull}(0, e_{ij}^- \mid e_{ij}^- \in \overline{\mathcal{V}}_F)\subset \mathcal{P}(A_n^+) \cap \mathcal{C}(F),$ $F \in \mathcal{F}$. Since if $x=(x_1, \ldots, x_{n+1})$ is in the facet of   $\sigma(\overline{F})$ opposite the origin, then  $|x_1|+\cdots+|x_{n+1}|=2$ and for any point  $x=(x_1, \ldots, x_{n+1}) \in \mathcal{P}(A_n^+) $, $|x_1|+\cdots+|x_{n+1}|\leq 2$ it follows that  $\mathcal{P}(A_n^+) \cap \mathcal{C}(F)\subset \sigma(\overline{F})$. Thus, $\sigma(\overline{F})=\mathcal{P}(A_n^+) \cap \mathcal{C}(F)$.  Finally, $\textrm{ConvHull}(0, e_{ij}^- \mid e_{ij}^- \in \overline{\mathcal{V}}_G)=\sigma(\overline{G} )=\cup_{F \in \mathcal{F}}\sigma(F)=\cup_{F \in \mathcal{F}}\sigma(\overline{F})
=\cup_{F \in \mathcal{F}}(\mathcal{P}(A_n^+) \cap \mathcal{C}(F))=\mathcal{P}(A_n^+) \cap (\cup_{F \in \mathcal{F}}\mathcal{C}(F)) = \mathcal{P}(A_n^+) \cap \mathcal{C}(G)$ as desired.

\qed

    \begin{lemma} \label{cone_reduction_lemma} \textbf{(Cone Reduction Lemma)} 
Given   a graph $G_0 \in \mathcal{L}_n$ with $d$ edges, let   $G_1, G_2, G_3$ be the graphs described as by  equations (\ref{graphs}).   Then  $G_1, G_2, G_3 \in \mathcal{L}_n$,
$$\mathcal{C}(G_0)=\mathcal{C}(G_1) \cup \mathcal{C}(G_2)$$   where all cones  $\mathcal{C}(G_0), \mathcal{C}(G_1), \mathcal{C}(G_2)$ are   $d$-dimensional and    
$$\mathcal{C}(G_3)=\mathcal{C}(G_1) \cap \mathcal{C}(G_2)  \mbox{   is $(d-1)$-dimensional. } $$
\end{lemma}

    \proof Let the edges of $G_0$ be 
 $f_1=(i, j), f_2=(j, k), f_3, \ldots, f_{d}$.   Let $\v(f_1),$ $\v(f_2),$ $\v(f_3), \ldots, \v(f_d)$  denote the vectors the edges of $G_0$ correspond to under the correspondence $\v: (i, j) \mapsto e_{ij}^-$, where $i<j$. Since  $G_0 \in \mathcal{L}_n$, the vectors  $\v(f_1), \v(f_2), \v(f_3), \ldots, \v(f_d)$ are linearly independent. By equations (\ref{graphs}),  $\mathcal{C}(G_0)=  \langle \v(f_1), \v(f_2), \v(f_3), \ldots, \v(f_d) \rangle $,   $\mathcal{C}(G_1)=  \langle \v(f_1), \v(f_1)+\v(f_2), \v(f_3), \ldots, \v(f_d) \rangle$,    $\mathcal{C}(G_2)=  \langle \v(f_1)+\v(f_2), \v(f_2),$ $ \v(f_3), \ldots, \v(f_d) \rangle$, 
 
 \noindent $\mathcal{C}(G_3)=\langle \v(f_1)+\v(f_2), \v(f_3), \ldots, \v(f_d) \rangle$. Thus,   $G_1, G_2, G_3 \in \mathcal{L}_n$, cones $\mathcal{C}(G_0),\mathcal{C}(G_1)$ and $\mathcal{C}(G_2)$ are $d$-dimensional, while cone $\mathcal{C}(G_3)$ is $(d-1)$-dimensional. 
 
 Clearly, $\mathcal{C}(G_1) \cup \mathcal{C}(G_2) \subset \mathcal{C}(G_0)$. Any vector $v \in \mathcal{C}(G_0)$   expressed in the basis $\v(f_1), \v(f_2), \v(f_3), \ldots, \v(f_d)$ satisfies either $[\v(f_1)]v \geq [\v(f_2)]v$ or $[\v(f_1)]v < [\v(f_2)]v$. Thus,  if $ v \in \mathcal{C}(G_0)$, then $v \in \mathcal{C}(G_1)$ or $v \in   \mathcal{C}(G_2)$.  Therefore,  $\mathcal{C}(G_0)=\mathcal{C}(G_1) \cup \mathcal{C}(G_2)$.

Clearly,  $\mathcal{C}(G_3) \subset \mathcal{C}(G_1) \cap \mathcal{C}(G_2)$.  Any $v \in \mathcal{C}(G_1)$ expressed in the basis $\v(f_1), \v(f_2), \v(f_3), \ldots, \v(f_d)$ satisfies $[\v(f_1)]v \geq [\v(f_2)]v$, while  $v \in \mathcal{C}(G_2)$ expressed in the basis $\v(f_1), \v(f_2), \v(f_3), \ldots, \v(f_d)$ satisfies   $[\v(f_1)]v \leq [\v(f_2)]v$. Thus,  $v \in \mathcal{C}(G_1)\cap  \mathcal{C}(G_2)$ expressed in the basis $\v(f_1), \v(f_2), \v(f_3), \ldots, \v(f_d)$ satisfies  $[\v(f_1)]v =[\v(f_2)]v$. Therefore, $ \mathcal{C}(G_1) \cap \mathcal{C}(G_2) \subset \mathcal{C}(G_3)$, leading to $ \mathcal{C}(G_1) \cap \mathcal{C}(G_2) = \mathcal{C}(G_3) $.
\qed

\medskip
  
         \noindent \textit{Proof of the  Reduction Lemma (Lemma \ref{reduction_lemma}).} Straightforward corollary of Lemmas \ref{equivalent} and \ref{cone_reduction_lemma}.\qed

\medskip
                       
    In Section \ref{sec:forest} we use  Lemmas \ref{reduction_lemma} and \ref{simplex} to prove general theorems about acyclic root polytopes, which can be specialized to yield proofs of parts of  Theorems \ref{vol} and \ref{thm1}.

\section{General theorems for acyclic root polytopes}
\label{sec:forest}

In this section we prove general theorems about acyclic root polytopes and    reduced forms of monomials $m^\mathcal{S}[F]$, for a forest $F$.

Given a polytope $\mathcal{P}\subset \mathbb{R}^{n+1}$, the {\bf $t^{th}$ dilate} of $\mathcal{P}$ is 
$$\displaystyle t \mathcal{P}=\{(tx_1, \ldots, tx_{n+1}) |  (x_1, \ldots, x_{n+1}) \in \mathcal{P}\}.$$

The {\bf Ehrhart polynomial of an integer polytope} $\mathcal{P}\subset \mathbb{R}^{n+1}$ is   
$$\displaystyle L_{\mathcal{P}} (t) = \# (t\mathcal{P} \cap \mathbb{Z}^{n+1}).$$

 For background on the  theory of Ehrhart  polynomials see \cite{br}.

\begin{lemma} \label{const}  Let $\mathcal{P}(G)^\circ=\bigsqcup_{\sigma^\circ \in  S} \sigma^\circ$, where $S$ is a collection of open simplices $\sigma^\circ$, such that  the origin is a vertex of each simplex in $S$ and the other vertices are from $\Phi^+$. 
Then the number of $i$-dimensional open simplices in $S$, denoted by $f_i$,   only depends on $\mathcal{P}(G)$, not on $S$ itself. 
\end{lemma}

\proof
Since   $ \displaystyle \mathcal{P}(G)^\circ= \bigsqcup _{\sigma^\circ \in S} \sigma^\circ$, we have that  
$ \displaystyle  L_{\mathcal{P}(G)^\circ}(t)=\sum _{\sigma^\circ \in S} L_{\sigma^\circ}(t)$. Since the vectors in $\Phi^+$ are unimodular, it follows that for a $d$-dimensional simplex $\sigma^\circ \in S$,  $ L_{\sigma^\circ}(t)=L_{\Delta^\circ}(t)$, where $\Delta$ is the standard $d$-simplex. By \cite[Theorem 2.2]{br} $L_{\Delta^\circ}(t)= {t-1 \choose d}.$
Thus,  
$$ L_{\mathcal{P}(G)^\circ}(t)=\sum_{i=0}^{\infty} f_i  {t-1 \choose i},$$ where $ L_{\mathcal{P}(G)^\circ}(t)\in \mathbb{Z}[t]$ and the set $\{ {t-1 \choose i} \mid i=0, 1, \ldots \}$ is a basis of $\mathbb{Z}[t]$. Therefore, $f_i$  are uniquely determined for  $i=0, 1, \ldots $, by  $\mathcal{P}(G)$ and are  independent of $S$. \qed

\begin{theorem} \label{thm_forest}
Let $F$ be any forest on the vertex set $[n+1]$ with $l$ edges. If $\mathcal{T}_F^\mathcal{S}$ is an $\mathcal{S}$-reduction tree with root labeled $F$, then the number of leaves of $\mathcal{T}_F^\mathcal{S}$ labeled by forests with $k$ edges,  denoted by $f_{F, k}$, is a function of $F$ and $k$ only.

In other words, 
if $P^\mathcal{S}_n$ is a reduced form of $m^\mathcal{S}[F]$, then $$P^\mathcal{S}_n(x_{ij}=1)=\sum_{l=0}^{l-1}f_{F, l-m} \beta^m.$$

\end{theorem}

\proof Let $\mathcal{T}_F^\mathcal{S}$  be a particular $\mathcal{S}$-reduction tree with root labeled $F$. By definition,  the leaves of $\mathcal{T}_F^\mathcal{S}$  are labeled by alternating forests with $k$ edges, where $k \in [l]$.  Let the $c_k$ forests $F^k_1, \ldots, F^k_{c_k}$ label the  leaves  of $\mathcal{T}_F^\mathcal{S}$ with $k$ edges, $k \in [l]$. Repeated use of the Reduction Lemma (Lemma \ref{reduction_lemma}) implies that \begin{equation} \label{open} \mathcal{P}(F)^\circ=\bigsqcup_{k \in [l], i_k \in [c_k]}\mathcal{P}(F^k_{i_k})^\circ,\end{equation} where the right hand side is a disjoint union of simplices by Lemma \ref{simplex}. By Lemma \ref{const}, the number of $k$-dimensional simplices among 
 $\bigcup_{k \in [l], i_k \in [c_k]}\{\mathcal{P}(F^k_{i_k})^\circ\}$ is independent of  the particular $\mathcal{S}$-reduction tree $\mathcal{T}_F^\mathcal{S}$. Thus, $f_{T, k}=c_k$  only depends on $F$ and $k$.

The formula for the reduced form of $m^\mathcal{S}[F]$ evaluated at $x_{ij}=1$ follows from the correspondence between the leaves of $\mathcal{T}_F^\mathcal{S}$  and reduced forms described in Section \ref{sec:red}.

\qed

We  easily obtain the Ehrhart polynomial, and thus also the volume of the polytope ${\mathcal{P}(F) }$ with the techniques used above.

  \begin{theorem} \label{ehr_forest}   The Ehrhart polynomial of the polytope ${\mathcal{P}(F) }$, where $F$ is  a forest on the vertex set $[n+1]$ with $l$ edges, is 
 $$ L_{\mathcal{P}(F) }(t)= (-1)^l\sum_{i=0}^{l} (-1)^i f_{F, i} {t+i \choose i},$$
 
\noindent  where $f_{F, k}$ is the number of  leaves of $\mathcal{T}_F^\mathcal{S}$ labeled by forests with $k$ edges.  
 \end{theorem}
 
 \begin{proof}
It follows from the proofs of Lemma \ref{const} and Theorem \ref{thm_forest} that  $$ L_{\mathcal{P}(F)^\circ }(t)=  \sum_{i=0}^{l} f_{F, i} {t-1 \choose i}.$$ Since by the Ehrhart-Macdonald reciprocity \cite[Theorem 4.1]{br} $$L_{\mathcal{P}(F)}(t)=(-1)^{\dim \mathcal{P}(F)} L_{\mathcal{P}(F)^\circ}(-t),$$
 it follows that $$L_{\mathcal{P}(F)}(t)=(-1)^l \sum_{i=0}^{l} f_{F, i} {-t-1 \choose i}=(-1)^l\sum_{i=0}^{l} (-1)^i f_{F, i} {t+i \choose i}.$$
 \end{proof}

\begin{corollary} \label{vol_forest}
If $F$ is  a forest on the vertex set $[n+1]$ with $l$ edges, then $$\mbox{\rm vol}\,  \mathcal{P}(F) =\frac{f_{F, l}}{l!}.$$
\end{corollary}

\proof By \cite[Lemma 3.19]{br} the leading coefficient of $L_{\mathcal{P}(F)}(t)$ is equal to $\mbox{\rm vol} \, \mathcal{P}(F)$. We   also   obtain $\mbox{\rm vol}\, \mathcal{P}(F) =\frac{f_{F, l}}{l!}$ directly from the Reduction Lemma if we count the $l$-dimensional simplices in the triangulation of $ \mathcal{P}(F) $.

\qed

\section{Reductions in the noncommutative case}
\label{reductionsB}

 In this section we prove  two crucial lemmas about  reduction (\ref{red}) in the noncommutative case necessary for proving Theorem \ref{thm1}. While in the commutative case   reductions on $G^\mathcal{S}[m]$ could result in crossing graphs, we prove that in the noncommutative case exactly those reductions from the commutative case are allowed which result in no crossing graphs, provided that $m=m^\mathcal{B}[T]$ for a noncrossing tree $T$ with suitable edge labels specified below. Furthermore, we also show that if there are any two edges
 $(i, j)$ and $(j, k)$ with $i<j<k$  in a successor of $G^\mathcal{B}[m]$, then after suitably many commutations it is possible to apply reduction (\ref{red}). Thus, once the reduction  process terminates, the set of graphs obtained as leaves of the reduction tree are  alternating forests. Now, unlike in the commutative case, they are also noncrossing. In fact,     each noncrossing alternating spanning forest of $\overline T$   satisfying certain additional technical conditions  occurs among the leaves of the reduction tree exactly once, yielding a complete combinatorial description of the reduced form of $m^\mathcal{B}[T]$.

In terms of graphs the partial  commutativity   means that if $G$ contains two edges $(i, j)_a$ and $(k, l)_{a+1}$ with $i, j, k, l$ distinct, then we can replace these edges by $(i, j)_{a+1}$ and $(k, l)_{a}$, and vice versa. Reduction rule (\ref{red}) on the other hand means that if there are two edges $(i, j)_a$ and $(j, k)_{a+1}$ in $G_0$, $i<j<k$,  then we replace $G_0$ with three graphs $G_1, G_2, G_3$ on the vertex set $[n+1]$ and edge sets 

\begin{eqnarray} \label{graphs2}
E(G_1)&=&E(G_0)\backslash \{(i, j)_a\}\backslash \{(j, k)_{a+1}\}\cup \{(i, k)_a\}\cup \{(i, j)_{a+1}\} \nonumber \\
E(G_2)&=&E(G_0)\backslash \{(i, j)_a\}\backslash \{(j, k)_{a+1}\}\cup \{(j, k)_a\}\cup \{(i, k)_{a+1}\}  \nonumber \\E(G_3)&=&(E(G_0)\backslash \{(i, j)_a\}\backslash \{(j, k)_{a+1}\})^a\cup \{(i, k)_a\},
\end{eqnarray}

 \noindent where $(E(G_0)\backslash \{(i, j)_a\}\backslash \{(j, k)_{a+1}\})^a$ denotes the edges obtained from the edges $E(G_0)\backslash \{(i, j)_a\} \backslash \{(j, k)_{a+1}\}$ by reducing the label of each edge which has label greater than $a$ by $1$.

A \textbf{$\mathcal{B}$-reduction tree}  $\mathcal{T}^\mathcal{B}$ is defined analogously to an $\mathcal{S}$-reduction tree, except we use equation (\ref{graphs2}) to describe the children. See Figure \ref{fig} for an example.
 A graph $H $ is called a \textbf{$\mathcal{B}$-successor} of $G$ if it is obtained by a series of reductions from $G$.   For convenience, we refer to commutativity of $x_{ij}$ and $x_{kl}$ for distinct $i, j, k, l$  as \textbf{reduction (2)}, by which we mean the rule $x_{ij} x_{kl} \leftrightarrow x_{kl}x_{ij}$, for $i, j, k, l$ distinct, or, in the language of graphs, exchanging edges  $(i, j)_a$ and $(k, l)_{a+1}$ with  $(i, j)_{a+1}$ and $(k, l)_{a}$ for $i, j, k, l$  distinct.


A forest  $H$ on the vertex set $[n+1]$ and $k$ edges labeled   $1, \ldots, k$  is \textbf{good} if it satisfies the following conditions:

 $(i)$ If edges $(i, j)_a$ and $(j, k)_b$ are in $H$, $i<j<k$, then $a<b$.

$(ii)$ If edges $(i, j)_a$ and $(i, k)_b$ in $H$ are such that $j < k$, then $a>b$.

$(iii)$  If edges $(i, j)_a$ and $(k, j)_b$ in $H$ are such that $i < k$, then $a>b$.

$(iv)$ $H$ is noncrossing. 

No  graph  $H$ with a  cycle could satisfy all of $(i), (ii), (iii), (iv)$ simultaneously, which is why we only define good forests. Note, however, that any forest $H$ has an edge-labeling that makes it a good forest. 

\begin{lemma} \label{huh}
If the root of an $\mathcal{B}$-reduction  tree is labeled by   a good forest, then all nodes of it  are  also labeled by  good forests.  
\end{lemma}

  \proof
The root of the $\mathcal{B}$-reduction tree is trivially labeled by a good forest.  We show that after each reduction (\ref{red}) or (2) all properties $(i), (ii), (iii), (iv)$ of good forests are preserved.

In reduction (2) we take disjoint edges  $(i, j)_a$ and $(k, l)_{a+1}$ and replace them by the edges  $(i, j)_{a+1}$ and $(k, l)_a$. It is easy to check that properties  $(i), (ii), (iii), (iv)$ are preserved using the fact that all edge-labels are integers and are not repeated, so the relative orders of edge-labels for edges incident to the same vertex are unchanged.   

Performing  reduction (\ref{red}) results in three new graphs as described by equation (\ref{graphs2}).  It is easy to check that properties $(i), (ii), (iii)$ are preserved using the fact that all edge-labels are integers and are not repeated. To prove that property $(iv)$ is also preserved, note that by $(i), (ii), (iii)$  if edge $(i, j)$ is labeled $a$ and $(j, k)$ labeled $a+1$, then there cannot be edges with endpoint $j$ of the form $(i_1, j)$ with $i_1<i$ or $(j, k_1)$ with $k<k_1$, or else some of the conditions $(i), (ii), (iii) $ would be violated. That there is no edge of the form described in the previous sentence with  endpoint $j$   together with the fact that the graph $G$  we applied reduction (\ref{red}) to was noncrossing  implies that  edge $(i, k)$ does not cross any edges of $G$, and therefore the resulting graph is also noncrossing.

\qed

A reduction applied to a noncrossing graph $G$  is \textbf{noncrossing} if the graphs resulting from the reduction are also noncrossing.

The following is then an immediate corollary of Lemma \ref{huh}.

\begin{corollary}
If $G$ is a  good forest, then all reductions that can be applied to $G$ and its $\mathcal{B}$-successors   are  noncrossing.
\end{corollary}

\begin{lemma} \label{cross}
Let $G$ be a   good forest.  Let $(i, j)_a$ and $(j, k)_b$ with $i<j<k$ be edges of $G$ such that no edge of $G$ crosses $(i, k)$. Then after finitely many applications of reduction (2) we can apply  reduction (\ref{red}) to edges  $(i, j)$ and $(j, k)$.
\end{lemma}

\begin{proof}
By the definition of a good forest it follows that $a<b$. If $b=a+1$, then we are done. Otherwise, consider all edges $(l, m)_c$ such that $a<c<b$. Since $G$ is a good forest and  $(i, k)$ does not cross any edges of $G$, we find that for any such edge $(l, m)_c$ is  either  disjoint from edges $(i, j)_a$ and $(j, k)_b$, or else  $(l, m)_c=(i, m)_c$ or  $(l, m)_c=(l, k)_c$. Then reduction (2) can be applied to the edges    $(l, m)_c$ with $a<c<b$ until  either the edges labeled $a$ and $a+1$ or the edges labeled $b-1$ and $b$ are disjoint, in which case we can perform reduction (2) on these edges. Once this is done, the difference between the labels of the edges 
 $(i, j)$ and $(j, k)$ decreased, and we can repeat this process until this difference is   $1$, in which case reduction (\ref{red}) can be applied to them.  
\end{proof}

\begin{corollary} \label{NA}
If $F$ labels a leaf of a $\mathcal{B}$-reduction tree whose  root    is labeled by a    good forest, then  $F$ is a good noncrossing  alternating forest. 
\end{corollary}
 
\proof
By Lemma \ref{huh}, $F$ is a good forest. By definition of good,   it is also noncrossing. Lemma \ref{cross} implies that $F$ is alternating, or else reduction (\ref{red}) could be applied to it, and thus it would not label a leaf of a $\mathcal{B}$-reduction tree.
\qed

\section{Proof  of Theorems \ref{vol} and  \ref{thm1} in a special case}
\label{reductionsB1}
  
In this    section we  prove    Theorems \ref{vol} and  \ref{thm1}  for the special case where $T=P=([n+1], \{(i, i+1) \mid i \in [n]\}).$ We prove the general versions of the theorems in Section \ref{sec:gen}. 
 
Given a  noncrossing  alternating forest $F$ on the vertex set $[n+1]$ with $k$ edges, the \textbf{lexicographic order} on its edges is as follows. Edge $(i_1, j_1)$ is less than edge $(i_2, j_2)$ in the  lexicographic  order  if  $j_1>j_2$, or $j_1=j_2$ and $i_1>i_2$. The forest $F$ is said to have \textbf{lexicographic  edge-labels} if its edges are labeled with integers $1, \ldots, k$ such that if edge $(i_1, j_1)$ is less than edge $(i_2, j_2)$ in lexicographic  order, then the label of $(i_1, j_1)$ is less than the label of   $(i_2, j_2)$ in the usual order on the integers. Clearly, given any graph $G$ there is a unique edge-labeling of it which is lexicographic. Note that our definition of lexicographic is closely related to the conventional definition, but it is not exactly the same. For an example of lexicographic  edge-labels, see the graphs labeling the leaves of the $\mathcal{B}$-reduction tree in Figure \ref{fig}.

\begin{lemma} \label{labeling}
If a noncrossing  alternating forest $F$ is a $\mathcal{B}$-successor of  a good forest, then upon some number of reductions (2) performed on $F$,  it is possible to obtain a  noncrossing  alternating forest $F'$ with lexicographic  edge-labels.
\end{lemma}

\begin{proof}
If edges  $e_1$ and $e_2$ of $F$   share a vertex and if $e_1$ is less than $e_2$ in the lexicographic order, then the label of $e_1$ is less than the label of $e_2$ in the usual order on integers by Lemma \ref{huh}. Since reduction (2)  swaps the labels of two vertex disjoint edges labeled by consecutive integers in  a graph, these swaps do not affect the relative order of the labels on edges sharing vertices.  Continue these swaps until the lexicographic order is obtained. 
\end{proof}

To avoid confusion about whether the commutative or the noncommutative version of the problem is being considered, we  denote  $x_{12}x_{23}\cdots x_{n,n+1}$  by $w_{\mathcal{S}}$ in the commutative and by $w_{\mathcal{B}}$ in the noncommutative case.

\begin{proposition} \label{non}
By choosing the series of reductions suitably, the set of leaves of a $\mathcal{B}$-reduction tree with root labeled by  $G^\mathcal{B}[w_{\mathcal{B}}]$ can be all noncrossing alternating forests $F$ on the vertex set  $[n+1]$  containing edge $(1, n+1)$ with lexicographic  edge-labels.
\end{proposition}

\begin{proof}
By Corollary \ref{NA}, all leaves of a $\mathcal{B}$-reduction tree are noncrossing alternating forests   on the vertex set  $[n+1]$. It is  easily seen that they all   contain edge $(1, n+1)$.
By the correspondence between the leaves of a $\mathcal{B}$-reduction tree    and simplices in a subdivision of  $\mathcal{P}(G^\mathcal{B}[w_{\mathcal{B}}])$ obtained from the Reduction Lemma (Lemma \ref{reduction_lemma}), it follows that no forest appears more than once among the leaves. Thus,   it suffices to prove that any  noncrossing alternating forest $F$ on the vertex set $[n+1]$ containing edge $(1, n+1)$ appears among the leaves of a $\mathcal{B}$-reduction tree and that all these forests have lexicographic  edge-labels. One can construct such a   $\mathcal{B}$-reduction tree by induction on $n$. 
 We  show that starting with the path $(1, 2), \ldots, (n, n+1)$ and performing reductions (1) and (2) we can obtain any  noncrossing alternating forest $F$ on the vertex set $[n+1]$  containing edge $(1, n+1)$ with lexicographic edge-labels.

 First perform the reductions on the path  $(1, 2), \ldots, (n, n+1)$ without involving edge $(n, n+1)$ in any of the reductions, until possible. Then we arrive to a set of trees where we have a   noncrossing alternating forest $F$ on the vertex set $[n]$  containing edge $(1, n)$ with lexicographic labeling and in addition edge $(n, n+1)_n$. By inspection  it follows that any  noncrossing alternating forest $F$ on the vertex set $[n+1]$  containing edge $(1, n+1)$ with lexicographic edge-labels can be obtained from them.
\end{proof}

\begin{theorem} \label{ajaj}
The set of leaves of a $\mathcal{B}$-reduction tree  with root labeled by  $G^\mathcal{B}[w_{\mathcal{B}}]$ is, up to applications of reduction (2),  the set of all  noncrossing alternating forests with lexicographic  edge-labels  on the vertex set $[n+1]$ containing edge $(1, n+1)$.
\end{theorem}

\proof

By Proposition \ref{non} there exists a $\mathcal{B}$-reduction tree which satisfies the conditions above. By Theorem \ref{thm_forest}  the number of forests with a fixed number of  edges among the leaves of an $\mathcal{S}$-reduction tree is independent of the particular $\mathcal{S}$-reduction tree, and, thus, the same is true for a $\mathcal{B}$-reduction tree.  It is clear that all forests labeling the leaves of a $\mathcal{B}$-reduction tree  with root labeled by $G^\mathcal{B}[w_{\mathcal{B}}]$ have to contains the edge $(1, n+1)$.  Also,   no vertex-labeled forest, with edge-labels disregarded, can appear twice among the leaves of a $\mathcal{B}$-reduction tree. Together with Lemma \ref{labeling} these imply the statement of Theorem \ref{ajaj}.
\qed

\medskip 

As corollaries of  Theorem \ref{ajaj} we obtain  the  characterziation of reduced forms of the noncommutative monomial $w_{\mathcal{B}}$, as well as a way to calculate $f_{P, k}$, the number of forests with $k$ edges labeling the leaves of an $\mathcal{S}$-reduction tree $\mathcal{T}_P^\mathcal{S}$ with root labeled $P=([n+1], \{(i, i+1) \mid i \in [n]\}).$

\begin{theorem} \label{main}
If  the polynomial $P^\mathcal{B}_n(x_{ij})$ is a reduced form of  $w_{\mathcal{B}}$,  then

$$P^\mathcal{B}_n(x_{ij})=\sum_F \beta^{n-|E(F)|} x^F,$$

\noindent where the sum runs over all   noncrossing alternating forests $F$ with lexicographic  edge-labels  on the vertex set $[n+1]$ containing edge $(1, n+1)$, and $x^F$ is defined to be the noncommutative monomial $\prod_{l=1}^k x_{i_l,j_l}$ if $F$ contains the edges $(i_1, j_1)_1, \ldots, (i_k, j_k)_k$.  

\end{theorem}

\begin{proposition}
The number of forests with $k$ edges labeling the leaves of an $\mathcal{S}$-reduction tree $\mathcal{T}_P^\mathcal{S}$, $f_{P, k}$,  is equal to the number of  noncrossing alternating forests on the vertex set  $[n+1]$ and $k+1$ edges such that edge $(1, n+1)$ is present. 
\end{proposition}

 \proof
 Theorem \ref{thm_forest} proves that  number of leaves  labeled by forests with $k$ edges in any $\mathcal{S}$-reduction tree with root labeled $P$ is independent of the particular $\mathcal{S}$-reduction tree. Since a $\mathcal{B}$-reduction tree becomes an $\mathcal{S}$-reduction tree when the edge-labels from the graphs labeling its nodes are deleted, the  number of leaves  labeled by forests with $k$ edges in any $\mathcal{S}$-reduction tree with root labeled $P$ is equal to the number of noncrossing alternating forests with lexicographic  edge-labels  on the vertex set $[n+1]$ with $k$ edges  containing edge $(1, n+1)$ by Theorem \ref{ajaj}. 
 
 \qed
 
  The {\bf Schr\"oder numbers} $s_n$ count the number of ways  to draw any number of diagonals of a convex $(n+2)$-gon that do not intersect in their interiors. Let   $s_{n, k}$ denote  the number   of ways  to draw $k$  diagonals of a convex $(n+2)$-gon that do not intersect in their interiors. Cayley  \cite{cay} in 1890 showed that   $\displaystyle s_{n, k}=\frac{1}{n+1} {n+k+1 \choose n} {n-1 \choose k}$.  

 \begin{lemma} \label{bij}  
 There is a bijection between the set of  noncrossing alternating forests on the vertex set  $[n+1]$ and $k+1$ edges such that edge $(1, n+1)$ is present and ways  to draw $k$  diagonals of a convex $(n+2)$-gon that do not intersect in their interiors.  Thus, $f_{P, k+1}=s_{n, k}$.
 \end{lemma}
 
 \proof
 The bijection can be described as follows.  Given a forest $F$ with edges $(i_1, j_1), \ldots, (i_k, j_k)$, $(1, n+1)$, correspond to it an $(n+2)$-gon on vertices $1, \ldots, n+2$ in a clockwise order, with diagonals  $(i_1, j_1+1), \ldots, (i_k, j_k+1)$.  
 
 \qed 
 
Using     $f_{P, k+1}=\frac{1}{n+1} {n+k+1 \choose n} {n-1 \choose k}$  we  specialize Theorems \ref{thm_forest} and \ref{ehr_forest} to 
 Theorems \ref{main2} and \ref{ehr_A}.

 \begin{theorem} \label{main2}
If  the polynomial $P^\mathcal{S}_n(x_{ij})$ is a reduced form of  $w_{\mathcal{S}}$,  then

$$P^\mathcal{S}_n(x_{ij}=1)= \sum_{m=0}^{n-1}  s_{n, n-m-1} \beta^m,$$

\noindent where $\displaystyle s_{n, k}=\frac{1}{n+1} {n+k+1 \choose n} {n-1 \choose k}$ is the number of noncrossing alternating forests on the vertex set $[n+1]$ with  $k+1$ edges, containing edge $(1, n+1)$.

\end{theorem}

\begin{theorem} \label{ehr_A} (Cf. \cite[Exercise 6.31]{ec2},  \cite{fong}) The Ehrhart polynomial of the polytope ${\mathcal{P}(A_n^+) }$ is 
 $$ L_{\mathcal{P}(A_n^+) }(t)= \frac{(-1)^n}{n+1} \sum_{i=0}^{\infty} {n+i \choose n} {n-1 \choose i-1}  {-t-1 \choose i}.$$
 \end{theorem}

The   generating function $J(\mathcal{P}(A_n^+) , x)=1+\sum_{t=1}^\infty  L_{\mathcal{P}(A_n^+) }(t)x^t$   was previously   calculated  by different methods; see \cite[Exercise 6.31]{ec2},  \cite{fong}.
   
\section{Proof  of Theorems \ref{vol} and  \ref{thm1} in the general case}
  \label{sec:gen} 
  
  In this section we find an analogue of Theorem \ref{main} for any noncrossing good tree $T$, and using it calculate the numbers $f_{T, k}$. 
  Specializing Theorems \ref{thm_forest} and \ref{ehr_forest} to $T$, we   then   conclude the proofs of Theorems   \ref{vol} and  \ref{thm1}.

   Theorems  \ref{main} and \ref{main2} imply Theorem \ref{thm1} for the special case   $T=P=([n+1], \{(i, i+1) \mid i \in [n]\}).$  We  generalize 
   Theorems \ref{ajaj}, \ref{main} and \ref{main2}   to  monomials $m^\mathcal{B}[T]$, where $T$ is a good tree.  For this we need some technical definitions.

Consider a noncrossing tree $T$ on $[n+1]$. We define the \textbf{pseudo-components} of $T$ inductively. The unique simple path $P$ from $1$ to $n+1$ is a pseudo-component of $T$. The graph $T\backslash P$ is an edge-disjoint union of   trees $T_1, \ldots, T_k$, such that   if $v$ is a vertex of $P$ and $v \in T_l$, $l \in [k]$, then $v$ is either the minimal or maximal vertex of $T_l$ . Furthermore, there are no $k-1$ trees whose edge-disjoint  union is $T\backslash P$ and which satisfy  all the requirements stated above. The set of pseudo-components of $T$, denoted by $ps(T)$ is $ps(T)=\{P\} \cup ps(T_1)\cup \cdots \cup ps(T_k)$. A pseudo-component $P'$ is said to be on $[i, j]$, $i<j$  if it is a path with endpoints $i$ and  $j$. A pseudo-component $P' $ on $[i, j]$ is said to be a \textbf{left pseudo-component} of $T$  if there are no edges $(s, i) \in E(T)$ with $s<i$   and a \textbf{right pseudo-component}  if  if there are no edges $(j, s) \in E(T)$ with $j<s$.     See Figure \ref{fig:pseudo} for an example.
   
   \begin{figure}[htbp] 
\begin{center} 
\includegraphics[width=.7\textwidth]{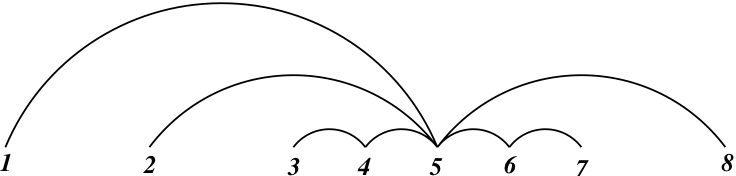} 
\caption{The edge sets of the pseudo-components in the graph depicted are $\{(1, 5), (5, 8)\}, \{(2, 5)\}, \{(3, 4), (4, 5)\}, \{(5, 6), (6, 7)\}$. The pseudo-component with edge set  $\{(1, 5), (5, 8)\}$ is a both a left and right pseudo-component, while the  pseudo-components with edge sets     $ \{(2, 5)\}, \{(3, 4), (4, 5)\}$     are  left  pseudo-components and the  pseudo-component with edge set  $\{(5, 6), (6, 7)\}$ is a  right pseudo-component.} 
\label{fig:pseudo}
\end{center} 
\end{figure}

\begin{proposition} \label{gen:non} Let $T$ be a good tree. 
By choosing the series of reductions suitably, the set of leaves of a $\mathcal{B}$-reduction tree with root $T$ can be all 
 noncrossing alternating spanning  forests of $\overline T$ with lexicographic  edge-labels  on the vertex set $[n+1]$ containing edge $(1, n+1)$ and at least one edge of the form $(i_1, j)$ with $i_1\leq i$ for each right  pseudo-component of $T$ on $[i, j]$ and  at least one edge of the form $(i, j_1)$ with $j\leq j_1$ for each left pseudo-component of $T$ on $[i, j]$. See Figure \ref{fig:pseudo-comp} for an example.
 \end{proposition}

    \begin{figure}[htbp] 
\begin{center} 
\includegraphics[width=.95\textwidth]{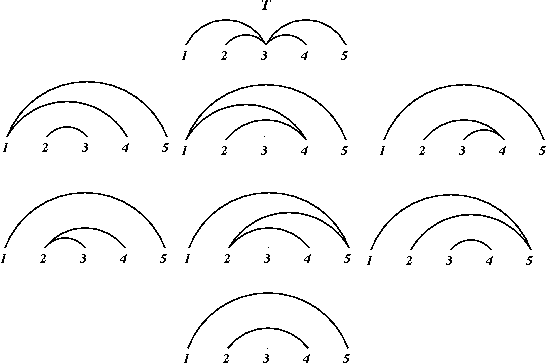} 
\caption{Trees $T_1, \ldots, T_6$ as depicted in Figure \ref{fig:lex} are the noncrossing alternating spanning trees of $\overline{T}$. This figure depicts all the other noncrossing alternating spanning  forests of $\overline T$ on the vertex set $[n+1]$ containing edge $(1, n+1)$ and at least one edge of the form $(i_1, j)$ with $i_1\leq i$ for each right  pseudo-component of $T$ on $[i, j]$ and  at least one edge of the form $(i, j_1)$ with $j\leq j_1$ for each left pseudo-component of $T$ on $[i, j]$. By the Ehrhart polynomial form of Theorem  \ref{vol}, see end of Section \ref{sec:gen}, $L_{\mathcal{P}(T)}(t)={t+2 \choose 2}-6 {t+3 \choose 3} + 6{t+4 \choose 4}$, since $f_{T, 2}=1, f_{T, 3}=6, f_{T, 4}=6$ and $f_{T, i}=0$, for $i \neq 2,3,4$.} 
\label{fig:pseudo-comp}
\end{center} 
\end{figure}


\begin{proof}  
 It is  easily seen that  all   graphs labeling the leaves of a  $\mathcal{B}$-reduction tree  must be noncrossing alternating spanning  forests of $\overline T$  on the vertex set $[n+1]$ containing edge $(1, n+1)$ and at least one edge of the form $(i_1, j)$ with $i_1\leq i$ for each right pseudo-component of $T$ on $[i, j]$ and  at least one edge of the form $(i, j_1)$ with $j\leq j_1$ for each left  pseudo-component of $T$ on $[i, j]$. The proof then follows the proof  of Proposition \ref{non}. To show that any noncrossing alternating spanning  forests of $\overline T$  on the vertex set $[n+1]$ containing edge $(1, n+1)$ and at least one edge of the form $(i_1, j)$ with $i_1\leq i$ for each right  pseudo-component of $T$ on $[i, j]$ and  at least one edge of the form $(i, j_1)$ with $j\leq j_1$ for each left  pseudo-component of $T$ on $[i, j]$  appears among the leaves of a $\mathcal{B}$-reduction tree and that all these forests have lexicographic  edge-labels, we use  induction on the number of pseudo-components of $T$. The base case is proved in Proposition \ref{non}. Suppose now that $T$ has $p$ pseudo-components, and let $P$ be such a pseudo-component that $T\backslash P$ is a tree with $p-1$ pseudo-components. Apply the inductive hypothesis to $T\backslash P$ and  Proposition \ref{non} to $P$ and combine the graphs obtained as outcomes in all the ways possible to obtain a set $S$ of graphs labeling the nodes of the reduction tree from which any leaf can be obtained by successive reductions. By inspection we  see that any  noncrossing alternating spanning  forest of $\overline T$  on the vertex set $[n+1]$ containing edge $(1, n+1)$ and at least one edge of the form $(i_1, j)$ with $i_1\leq i$ for each right pseudo-component of $T$ on $[i, j]$ and  at least one edge of the form $(i, j_1)$ with $j\leq j_1$ for each left  pseudo-component of $T$ on $[i, j]$ can be obtained by reductions from the elements of $S$. Since no graph can be obtained twice, and no other graph can label a leaf of a $\mathcal{B}$-reduction, the proof is complete. 
\end{proof}

\begin{theorem} \label{gen:ajaj} Let $T$ be a good tree. 
The set of leaves of a $\mathcal{B}$-reduction tree  with root labeled $T$ is, up to applications of reduction (2),  the set of all   noncrossing alternating spanning  forests of $\overline T$ with lexicographic edge-labels  on the vertex set $[n+1]$ containing edge $(1, n+1)$ and at least one edge of the form $(i_1, j)$ with $i_1\leq i$ for each right pseudo-component of $T$  on $[i, j]$ and  at least one edge of the form $(i, j_1)$ with $j\leq j_1$ for each left pseudo-component of $T$ on $[i, j]$.
 
\end{theorem}

\proof The proof is analogous to that of Theorem \ref{ajaj} using Proposition \ref{gen:non} instead of   Proposition \ref{non}.

\qed

\medskip 

As corollaries of  Theorem \ref{gen:ajaj} we obtain  the  characterziation of reduced forms of the noncommutative monomial $m^\mathcal{B}[T]$ for a good tree $T$, as well as a combinatorial description of  $f_{T, k}$, the number of forests with $k$ edges labeling the leaves of an $\mathcal{S}$-reduction tree $\mathcal{T}_T^\mathcal{S}$ with root labeled $T.$

\medskip

\noindent {\bf Theorem \ref{thm1}.}
{\it {\bf (Noncommutative part.)} If  the polynomial $P^\mathcal{B}_n(x_{ij})$ is a reduced form of  $m^\mathcal{B}[T]$ for a good tree $T$,  then

$$P^\mathcal{B}_n(x_{ij})=\sum_F \beta^{n-|E(F)|} x^F,$$

\noindent where the sum runs over all    noncrossing alternating spanning  forests of $\overline T$ with lexicographic edge-labels  on the vertex set $[n+1]$ containing edge $(1, n+1)$ and at least one edge of the form $(i_1, j)$ with $i_1\leq i$ for each right pseudo-component of $T$ on $[i, j]$ and  at least one edge of the form $(i, j_1)$ with $j\leq j_1$ for each left pseudo-component of $T$ on $[i, j]$, and $x^F$ is defined to be the noncommutative monomial $\prod_{l=1}^k x_{i_l,j_l}$ if $F$ contains the edges $(i_1, j_1)_1, \ldots, (i_k, j_k)_k$.  

}

\begin{proposition} \label{f_{T, k}} Let $T$ be a good tree. 
The number of forests with $k$ edges labeling the leaves of an $\mathcal{S}$-reduction tree $\mathcal{T}_T^\mathcal{S}$ with root labeled by $T$, $f_{T, k}$,  is equal to the number of  noncrossing alternating spanning forests $F$ of $\overline T$   containing edge $(1, n+1)$ and at least one edge of the form $(i_1, j)$ with $i_1\leq i$ for each right pseudo-component of $T$ on $[i, j]$ and  at least one edge of the form $(i, j_1)$ with $j\leq j_1$ for each left pseudo-component  of $T$ on $[i, j]$.
 
\end{proposition}
   
    Proposition \ref{f_{T, k}} provides a combinatorial description of the coefficients in Theorems \ref{thm_forest}, \ref{ehr_forest} and Corollary \ref{vol_forest}, completing  the proofs of Theorems \ref{vol} and \ref{thm1}. We state them in full generality here.

\medskip

\noindent {\bf Theorem \ref{thm1}.}
{\it {\bf (Commutative part.)}  If  the polynomial $P^\mathcal{S}_n(x_{ij})$ is a reduced form of  $m^\mathcal{S}[T]$ for a good tree $T$,  then 
 
  $$P^\mathcal{S}_n(x_{ij}=1)=\sum_{l=0}^{l-1}f_{T, l-m} \beta^m,$$

\noindent where $f_{T, k}$ is as in Proposition \ref{f_{T, k}}. }

\medskip

 \noindent {\bf Theorem \ref{vol}.}
{\it {\bf (Ehrhart polynomial and volume.)}     The Ehrhart polynomial and volume of the polytope ${\mathcal{P}(T) }$, for a good tree $T$  on the vertex set $[n+1]$, are, respectively, 
 
 $$ L_{\mathcal{P}(T) }(t)= (-1)^n\sum_{i=0}^{n} (-1)^i f_{T, i} {t+i \choose i},$$

 $$\mbox{\rm vol}\, \mathcal{P}(T) =\frac{f_{T, n}}{n!},$$

\noindent  where $f_{T, k}$ is as in Proposition \ref{f_{T, k}}. See Figure \ref{fig:pseudo-comp} for an example.
}

\medskip

Theorem \ref{vol} can be generalized so that we not only describe the $n$-dimensional simplices in the triangulation of  $\mathcal{P}(T)$, but also describe their intersections in terms of noncrossing alternating spanning forests in $\overline T$.  Using the Reduction Lemma (Lemma \ref{reduction_lemma}) and Theorem \ref{gen:ajaj} we can deduce the following.

\medskip

 \noindent {\bf Theorem \ref{vol}.}
{\it {\bf (Canonical triangulation.)}  If $T$ is a noncrossing tree on the vertex set $[n+1]$ and   $T_1,\ldots, T_k$ are the noncrossing alternating spanning trees of  $\overline T$, then the root polytopes $\mathcal{P}(T_1), \ldots, \mathcal{P}(T_k)$ are   $n$-dimensional simplices forming a triangulation of   $\mathcal{P}(T)$. Furthermore, the   intersections of the top dimensional simplices   $\mathcal{P}(T_1), \ldots, \mathcal{P}(T_k)$  are simplices $ \mathcal{P}(F)$, where  $F$ run over all   noncrossing alternating spanning  forests of $\overline T$ with lexicographic edge-labels  on the vertex set $[n+1]$ containing edge $(1, n+1)$and at least one edge of the form $(i_1, j)$ with $i_1\leq i$ for each right pseudo-component of $T$  on $[i, j]$ and  at least one edge of the form $(i, j_1)$ with $j\leq j_1$ for each left pseudo-component of $T$ on $[i, j]$. 
 }

 \section{Properties of the canonical triangulation}
 \label{sec:shelling}
 
 In this section we show that the \textbf{canonical triangulation} of  $\mathcal{P}(T)$ into simplices   $\mathcal{P}(T_1), \ldots, \mathcal{P}(T_k)$, and their faces, where $T_1,\ldots, T_k$ are the noncrossing alternating spanning trees of  $\overline T$, as described in Theorem \ref{vol}, is regular and flag.  We  construct a shelling and using this shelling  calculate the generating function  $J(\mathcal{P}(T) , x)=1+\sum_{t=1}^\infty  L_{\mathcal{P}(T) }(t)x^t$, yielding another way to compute the Ehrhart polynomials. This  generalizes the calculation of $J(\mathcal{P}(A_n^+) , x)$,  \cite[Exercise 6.31]{ec2}, \cite{fong}.


Recall that a triangulation of the polytope $P$ is {\bf regular} if there exists a concave  piecewise linear function $f:P \rightarrow \mathbb{R}$ such that the regions of linearity of $f$ are the maximal simplices in the triangulation. 
It has been shown in \cite[Theorem 6.3]{GGP} that the noncrossing triangulation of $\mathcal{P}(A_n^+)$ is regular. This result   can be naturally extended to the canonical triangulation of any of the root polytopes $\mathcal{P}(T)$. An attractive proof   uses the following concave function constructed by  Postnikov for an alternative proof of \cite[Theorem 6.3]{GGP}. 

 Let $f:A \rightarrow \mathbb{R}$ be a function on the set $A$ such that  polytope $P={\rm ConvHull}(A)$. 
Let $\tilde{P}={\rm ConvHull}((a, f(a)) \mid a \in A)$ and define then $f(p)={\rm max}\{x \mid \pi(a, x)=p, (a, x) \in \tilde{P}\}$, $p \in P$.   The  function $f:P \rightarrow \mathbb{R}$ is concave by definition. 
Consider the root polytope $\mathcal{P}(T)$ with  vertices $0$ and $e_i-e_j$, where $(i, j) \in I \times J$. 
Let $f(0)=0$ and $f(e_i-e_j)=(i-j)^2$ for $(i, j) \in I \times J$.  Extend this to a concave piecewise linear function as explained in the above paragraph. A check of the regions of linearity proves the regularity of the canonical triangulation of $\mathcal{P}(T)$.

 It can also be shown that the canonical triangulation of $\mathcal{P}(T)$ is flag, which we leave as an exercise to the reader. For the definition and importance of flag triangulations see \cite[Section 2]{h}.

 The canonical triangulation of  $\mathcal{P}(T)$ is  \textbf{shellable}, if there is a  \textbf{shelling}, a linear order $\v(f_1), \ldots, F_k$  on  $\mathcal{P}(T_1), \ldots, \mathcal{P}(T_k)$, such that for all $2 \leq i \leq k$, $F_i$ is attached to $ \v(f_1) \cup \ldots \cup F_{i-1}$ on a union of nonzero facets of $F_i$. See \cite{shelling} for more details.

 The {\bf lexicographic ordering} on the facets  $\mathcal{P}(T_1), \ldots, \mathcal{P}(T_k)$  is as follows: $\mathcal{P}(T_i)<_{lex} \mathcal{P}(T_j)$ if and only if for some $l$ the first $l$ edges of $T_i$ and $T_j$  in lexicographic ordering coincide and the $(l+1)^{st}$ edge of $T_i$ is less than the $(l+1)^{st}$ edge of $T_j$  in lexicographic ordering. The lexicographic ordering on the edges differs from the one we defined in Section \ref{reductionsB1}; instead, here we  use the conventional one. Namely, edge $(i_1, j_1)$ is less than edge $(i_2, j_2)$ in the  lexicographic  order  if  $i_1<i_2$, or $i_1=i_2$ and $j_1<j_2$.
 
 \begin{theorem} \label{shellable} Let $T$ be a noncrossing tree on the vertex set $[n+1]$. Let $T_1,\ldots, T_k$ be the noncrossing alternating spanning trees of  $\overline T$ such that 
  $ \mathcal{P}(T_1)<_{lex}\cdots<_{lex} \mathcal{P}(T_k)$. Then  $\mathcal{P}(T_1),$ $ \ldots, \mathcal{P}(T_k)$   is a shelling order. See Figure \ref{fig:lex} for an example. 
 
 \end{theorem}

 \proof It suffices to show that for all $2 \leq m \leq k$, the intersection $ \mathcal{P}(T_m)\cap (\mathcal{P}(T_1)\cup \ldots \cup  \mathcal{P}(T_{m-1}))$  is a  union of nonzero facets of $\mathcal{P}(T_m)$.

Let  $L(T_m)$ denote the set of left vertices of $T_m$, that is, the vertices of $T_m$ which are the smaller vertex of each edge incident to them. Let  $$S(T_m)=\{ (i, j) \mid  i \in L(T_m) \mbox{ and $j$ is the largest vertex adjacent to $i$ in $T_m$} \}.$$ The set $S(T_m)$ uniquely determines $T_m$, since $T_m$ is a noncrossing alternating spanning tree. 

There are exactly two  noncrossing alternating trees containing $F=([n+1], E(T_m)\backslash \{(i, j)\})$, $(i, j) \in S(T_m)\backslash \{(1, n+1)\}$, namely, $T_m$ and $\tilde{T}_m=([n+1], E(F)\cup \{(i', j')\}),$ where $i'$ is the biggest vertex of $T_m$ smaller than $i$ such that  $(i', j) \in E(T_m),$ and $j'$ is the biggest vertex of $T_m$ smaller than $j$ such that $(i, j') \in E(T_m)$, or if $(i, j)$ is the only edge incident to $i$, then $j'=i$. Let $f_{T_m}: S(T_m)\backslash \{(1, n+1)\} \rightarrow E(K_{n+1})$ be defined by  $f_{T_m}: (i, j) \mapsto (i', j')$ according to the rule explained above. 
Define $$M_T(T_m)=\{(i, j) \in S(T_m) \mid f_{T_m}((i, j)) \not \in \overline{T}\}.$$

The set $S_T(T_m)=S(T_m) \backslash M_T(T_m)$  uniquely determines $T_m$, since $T_m$ is a noncrossing alternating spanning tree of $\overline{T}$.  Furthermore, if for some $m' \in [k]$, $m' \neq m$,  $S_T(T_m)\subset E(T_{m'})$, then $\mathcal{P}(T_m)<_{lex} \mathcal{P}(T_{m'})$.  
Thus, if for a forest $F$ on the vertex set $[n+1]$,  $S_T(T_m)\subset E(F) \subset E(T_m)$, then $\mathcal{P}(F)$ is not a face of $\mathcal{P}(T_1)\cup \ldots \cup  \mathcal{P}(T_{m-1})$. 
If   $F\subset T_m$ does not contain $S_T(T_m)$ and  $|E(F)|=n-1$, then $F \subset T_l= ([n+1], E(T_m)\backslash \{(i, j)\} \cup \{f_{T_m}((i, j))\})$ for $l<m$. Thus,      for all $2 \leq m \leq k$, $$ \mathcal{P}(T_m)\cap (\mathcal{P}(T_1)\cup \ldots \cup  \mathcal{P}(T_{m-1}))= \bigcup_{(i, j)\in S_T(T_m)} \mathcal{P}(([n+1], E(T_m)\backslash \{(i, j)\})).$$

See Figure \ref{fig:lex} for an example. 
 \qed

    \begin{figure}[htbp] 
\begin{center} 
\includegraphics[width=.95\textwidth]{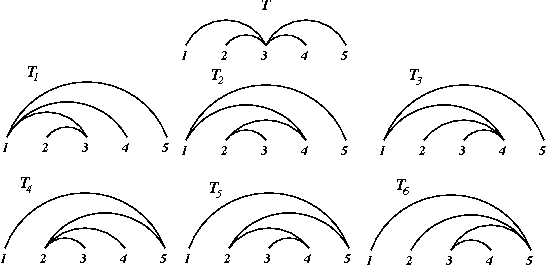} 
\caption{Trees $T_1, \ldots, T_6$ are the noncrossing alternating spanning trees of $\overline{T}$.
The root polytopes associated to them satisfy  $ \mathcal{P}(T_1)<_{lex}\cdots<_{lex} \mathcal{P}(T_6)$. 
\newline $S_T(T_1)=\emptyset, S_T(T_2)=\{(2,4)\}, S_T(T_3)=\{3,4\}, \newline S_T(T_4)=\{(2,5)\}, S_T(T_5)=\{(2,5), (3,4)\}, S_T(T_6)=\{(3,5)\}$.
\newline By Theorem \ref{J},  $J( \mathcal{P}(T), x)=\displaystyle\frac{x^2+4x+1}{(1-x)^{5}}. $ This is of course equivalent to $L_{\mathcal{P}(T)}(t)={t+2 \choose 2}-6 {t+3 \choose 3} + 6{t+4 \choose 4}$ as calculated in Figure \ref{fig:pseudo-comp}. For a way to see this equivalence directly, see \cite[Lemma 3.14]{br}.} 
\label{fig:lex}
\end{center} 
\end{figure}

 \begin{theorem} \label{J} Let $T$ be a good tree on the vertex set $[n+1]$. Let $c(n, l)$ be the number of noncrossing alternating spanning trees $T_m$ of $\overline{T}$ with $|S_T(T_m)|=l$.   Then, $$(1-x)^{n+1} J( \mathcal{P}(T), x)=\sum_{l=1}^n c(n, l-1) x ^{l-1}.$$
 
 \end{theorem}
 
 \proof It can be seen that for a forest $F$ with $r$ edges, $ J( \mathcal{P}(F), x)=\frac{1}{(1-x)^{r+1}},$  \cite[Theorem 2.2]{br}. If we are adding the simplices $\mathcal{P}(T_1), \ldots,  \mathcal{P}(T_{k})$ in lexicographic order one at a time, and calculating their contribution to $J( \mathcal{P}(T), x)$, then the contribution of  $\mathcal{P}(T_m)$ such that  $ \mathcal{P}(T_m)\cap (\mathcal{P}(T_1)\cup \ldots \cup  \mathcal{P}(T_{m-1}))$ is a  union of $(l-1)$  facets of $\mathcal{P}(T_m)$ is $$\frac{1}{(1-x)^{n+1}}-(l-1)\frac{1}{(1-x)^{n}}+ \cdots + (-1)^{l-1} {l-1 \choose l-1} \frac{1}{(1-x)^{n+1-(l-1)}}=\frac{x^{l-1}}{(1-x)^{n+1}}.$$ Hence, $$J( \mathcal{P}(T), x)=\frac{\sum_{l=1}^n c(n, l-1) x ^{l-1}}{(1-x)^{n+1} }.$$

 \qed
 
 \begin{remark}
All the theorems proved for trees (monomials corresponding to trees) in this paper  can be formulated for forests (monomials corresponding to forests), and the proofs proceed analogously. The acyclic condition for graphs in the theorems  is crucial for the proof techniques to work, but the noncrossing condition is not. Given an acyclic graph $G$ which is crossing, we can {\bf uncross} it to obtain a new graph $G^u$.  The graph $G^u$ is a noncrossing graph such that there is a graph  isomorphism $\phi: G \rightarrow G^u$, where if $(i, j) \in E(G)$, $i<j$, then $\phi(i)<\phi(j)$. The graph $G^u$ is not uniquely determined by these conditions.  
All the results apply to any  $G^u$, and they can be translated back for $G$ in an obvious way.  E.g.   the volume of $\mathcal{P}(T)$ for any tree $T$ on the vertex set $[n+1]$ is $\mbox{\rm  vol}\, \mathcal{P}(T)=f_{T^u} \frac{1}{n!},$ where  $f_{T^u} $ denotes the number of noncrossing alternating spanning trees of  $\overline{T^u}$, the transitive closure of the uncrossed $T$.
 \end{remark}

 \section{Unique reduced forms and Gr\"obner bases}
\label{sec:grobi}

The reduced form of a monomial $m \in \b$ was defined in the Introduction as  a polynomial $P_n^\mathcal{B}$ obtained by successive applications of the reduction rule  (\ref{red}) until no further reduction is possible, where we allow commuting   any two variables $x_{ij}$ and $x_{kl}$ where $i, j, k, l$ are distinct,   between the reductions.  An alternative way of thinking of  the reduced form of a monomial $m \in \b $ is to view the reduction process in $\q\langle \beta, x_{ij} \mid 1\leq i<j\leq n \rangle / I_{\beta},$ where the generators of the (two-sided) ideal $I_{\beta}$ in $\q\langle \beta, x_{ij} \mid 1\leq i<j\leq n+1 \rangle$  are the elements $x_{ij}x_{kl}-x_{kl}x_{ij}$ for $i<j, k<l$ distinct, and $\beta x_{ij}-x_{ij} \beta$, $i<j$.  In this section we prove the following theorem.

  \begin{theorem} \label{uu}
The reduced form of any monomial $m \in \b$ is unique. 
\end{theorem}

We use noncommutative Gr\"obner bases techniques, which we now briefly review.   
 We use the terminology and notation of \cite{g}, but state the results only for our special algebra. For the more general statements, see \cite{g}. Throughout this section  we consider the noncommutative case only. 
 
Let  $${ \bf R=\q\langle \beta, x_{ij} \mid 1\leq i<j\leq n+1 \rangle / I_{\beta}}$$ with multiplicative basis $\base$, the set of noncommutative monomials in variables $\beta$ and $ x_{ij}$,  where $1\leq i< j\leq n$,  up to equivalence under the commutativity relations described by $I_\beta$. 

The {\bf tip} of an element $f \in R$ is the largest basis element appearing in its expansion, denoted by Tip$(f)$. Let CTip$(f)$ denote the coefficient of Tip$(f)$ in this expansion. A set of elements $X$ is {\bf tip reduced} if for distinct elements $x, y \in X$, Tip$(x)$ does not divide Tip$(y)$.

\medskip

A well-order $>$ on $\base$ is {\bf admissible} if for $p, q, r, s \in \base$:

1. if $p<q$ then $pr<qr$ if both $pr\neq 0$ and $qr \neq 0$;

2. if $p<q$ then $sp<sq$ if both $sp\neq 0$ and $sq \neq 0$;

3. if $p=qr$, then $p>q$ and $p>r$.

\medskip

Let $f, g \in R$  and suppose that there are monomials $b, c \in \base$ such that 

\medskip

1. Tip$(f)c$=$b$Tip$(g)$.

\medskip

2. Tip$(f)$ does not divide $b$ and Tip$(g)$ does not divide $c.$

\medskip
Then the {\bf overlap relation of $f$ and $g$ by $b$ and $c$} is 

$$o(f, g, b, c)=\frac{fc}{\mbox{CTip}(f)}-\frac{bg}{\mbox{CTip}(g)}.$$


\medskip 
 \begin{proposition} \label{suf} (\cite[Theorem 2.3]{g}) A tip reduced  generating set of elements $\gr$ of the ideal $J$ of $R$ is a Gr\"obner basis, where the ordering on the monomials is admissible, if for every overlap relation $$o(g_1, g_2, p, q) \Rightarrow_{\gr} 0,$$

\noindent where $g_1, g_2 \in \gr$ and the above notation means that dividing $o(g_1, g_2, p, q)$ by $\gr$ yields a remainder of $0$. 
\end{proposition}
See \cite[Theorem 2.3]{g} for the more general formulation of Proposition \ref{suf} and  \cite[Section 2.3.2]{g} for the formulation of the Division Algorithm.

\begin{proposition} \label{grobi} Let $J$ be the ideal generated by the elements 
   $$x_{ij}x_{jk}-x_{ik}x_{ij}-x_{jk}x_{ik}-\beta x_{ik},\mbox{   for  }1\leq i<j<k\leq n+1,$$

in $R$. Then there is a monomial order in which the above generators of $J$ form a Gr\"obner basis $\gr$ of $J$ in $R$, and the tips of the generators are, $x_{ij}x_{jk}$.  
  
\end{proposition}

\proof Let   $x_{ij}>x_{kl}$   if $(i, j)$ is less than $(k, l)$ lexicographically. The degree of a monomial is determined by setting the degrees of $x_{ij} $ to be $1$ and the degrees of $\beta$ and scalars to be $0$. A monomial with higher degree is bigger in the order $>$, and the lexicographically bigger monomial of the same degree is greater than the lexicographically smaller one. Since in $R$ two equal monomials can be written in two different ways due to commutations, we can pick a representative to work with, say the one which is the ``largest" lexicographically among all possible ways of writing the monomial, to resolve any ambiguities. The order $>$ just defined is admissible, and  in it the tip of  
  $x_{ij}x_{jk}-x_{ik}x_{ij}-x_{jk}x_{ik}-\beta x_{ik}$,  for  $1\leq i<j<k\leq n+1$, is
  $x_{ij}x_{jk}$.  
In particular, the generators of $J$ are tip reduced. A calculation of the overlap relations shows that  $o(g_1, g_2, p, q) \Rightarrow_{\gr} 0$ in $R$,   where $g_1, g_2 \in \gr$. 
Proposition \ref{suf} then implies Proposition \ref{grobi}.
\qed

\begin{corollary} \label{/Y}
The reduced form of a noncommutative monomial $m $  in variables $\beta$ and $x_{ij}$, $1\leq i<j\leq n+1$, is unique in $R$. 
\end{corollary}

\proof Since the tips of elements of the Gr\"obner basis $\gr$ of  $J$ are exactly the monomials which we replace in the prescribed reduction rule   (\ref{red}), the reduced form of a monomial $m$ is the remainder $r$ upon division by the elements of $\gr$ with the order $>$ described in the proof of Proposition \ref{grobi}. Since we proved that in $R$ the basis $\gr$ is a Gr\"obner basis of $J$, it follows 
by \cite[Proposition 2.7]{g} that the remainder $r$ of the division of $m$ by $\gr$ is unique in $R$.  That is, the reduced form of a good monomial $m$  is unique in $R$. 
\qed

Note that Corollary \ref{/Y} is equivalent to Theorem \ref{uu}.

  \section*{Acknowledgement}
 I am grateful to my advisor Richard Stanley for suggesting  this problem  and for many helpful suggestions. I would  like to thank Alex Postnikov   for sharing his insight into root polytopes and for his encouragement. I would also like to thank    Anatol Kirillov for drawing my attention to the noncommutative side  of the problem.

\end{document}